\documentclass[a4paper,article]{memoir}
\usepackage{amsthm,amsmath,amssymb,amsfonts,slashed,hyperref}
\newcommand{\D}{\hskip 0.3mm \textup{d}}

\newcommand{\Ad}{\textup{Ad}}
\newcommand{\ad}{\textup{ad}}
\newcommand{\GL}{\textup{GL}}
\newcommand{\tr}{\textup{tr}}
\newcommand{\Div}{\textup{div}}
\newcommand{\const}{\textup{const}}
\newcommand{\hb}[1]{E_{#1}}
\newcommand{\bhk}{\Phi}
\newtheorem{thm}{Theorem}[chapter]
\newtheorem*{thm*}{Theorem}
\newtheorem{lem}[thm]{Lemma}
\newtheorem{prop}[thm]{Proposition}
\theoremstyle{definition}
\newtheorem{rmk}[thm]{Remark}
\setlrmarginsandblock{3.5cm}{3.5cm}{*}
\checkandfixthelayout
\numberwithin{equation}{chapter}
\title{Local monotonicity for the Yang-Mills-Higgs flow}
\author{Ahmad Afuni}
\date{\vskip -5ex}
\begin{document}
\maketitle
\begin{abstract}A local monotonicity formula for the Yang-Mills-Higgs flow on $G$-bundles over $\mathbb{R}^n$ ($n>4$) is proved. It is shown that the monotone quantity co\"incides on certain self-similar solutions with that appearing in existing non-local monotonicity formul\ae\ for the Yang-Mills and Yang-Mills-Higgs flows.
\end{abstract}
\chapter{Introduction}
This paper is concerned with local monotonicity properties satisfied by the Yang-Mills-Higgs flow on $\mathbb{R}^{n}$ with \textit{trivial} underlying principal bundle, where we assume throughout that $n>4$. We begin with a brief account of Yang-Mills-Higgs theory in this setting and refer the reader to \cite{MR614447} and the references therein for details.

Let $G$ be a connected compact finite-dimensional semisimple Lie group with Lie algebra $(\mathfrak{g},\left[\cdot,\cdot\right])$. Denote the \textit{adjoint representation} of $G$ on $\mathfrak{g}$ by $\Ad$ and that of $\mathfrak{g}$ on $\mathfrak{g}$, viz. the representation obtained by differentiating $\Ad$ at the identity, by $\ad$. The Lie algebra $\mathfrak{g}$ together with the \textit{negative Killing form}
\begin{align*}
	\mathfrak{g}\times\mathfrak{g}\ni A\times B\mapsto \left<A,B\right>:=-\tr\left(\ad_{A}\circ\ad_{B}\right)
\end{align*}
forms an inner product space and $\left<\cdot,\cdot\right>$ is $\Ad$-invariant, i.e. for all $g\in G$ and $A,B\in\mathfrak{g}$,
\begin{align*}
	\left<\Ad_{g}A,\Ad_{g}B\right>=\left<A,B\right>.
\end{align*}
In this setting, a \textit{gauge field} or \textit{connection} on $\mathbb{R}^n\times G$ is given by a $\mathfrak{g}$-valued one-form on $\mathbb{R}^n$ which we shall write as 
\begin{align*}
	A=\sum_{i=1}^{n}A_{i}\otimes \D x^{i}:\mathbb{R}^n\rightarrow \mathfrak{g}\otimes T^\ast \mathbb{R}^n
\end{align*}
and has \textit{field strength} or \textit{curvature} given by the $\mathfrak{g}$-valued two-form
\begin{align*}
	F=\sum_{i<j}F_{ij}\otimes \D x^i \wedge \D x^j,
\end{align*}
where $F_{ij}=\partial_i A_j - \partial_j A_i + [A_i , A_j]$. It is clear that $F_{ij}$ is antisymmetric and a straightforward computation shows that the \textit{Bianchi identity}
\begin{align}\label{bianchi}
	\nabla_i F_{jk} + \nabla_j F_{ki} + \nabla_{k}F_{ij}=0
\end{align}
holds, where $\nabla_i$ denotes the $i$th \textit{gauge-covariant partial derivative} of a $\mathfrak{g}$-valued function defined by
\begin{align*}
	\nabla_i = \partial_i + \left[A_i , \cdot\ \right]
\end{align*}
and $\partial_{i}$ is the usual $i$th partial derivative of a vector-valued function; this differential operator is \textit{compatible} with the inner product $\left<\cdot,\cdot\right>$ in the sense that for all $i\in\left\{1,\dots, n\right\}$ and $\mathfrak{g}$ valued functions $X$ and $Y$,
\begin{align}
	\partial_{i}\left<X,Y\right>=\left<\nabla_{i}X,Y\right> + \left<X,\nabla_{i}Y\right>.\label{gcompat}
\end{align}
We also suppose given a (smooth) representation $\rho:G\rightarrow \GL(V)$ on a finite-dimensional inner product space $(V,\left<\cdot,\cdot\right>)$ and assume that $\left<\cdot,\cdot\right>$ is \textit{$\rho$-invariant}. A \textit{scalar field} is a smooth function $u:\mathbb{R}^n\rightarrow V$; for such fields, the $i$th \textit{gauge-covariant partial derivative} is defined by
\begin{align*}
	\slashed{\nabla}_i u = \partial_i u + A_{i}\cdot u,
\end{align*}
where $\cdot$ is the action of $\mathfrak{g}$ on $V$ obtained by differentiating $\rho$ at the identity of $G$. This differential operator satisfies the identity
\begin{align}
	\left(\slashed{\nabla}_{i}\slashed{\nabla}_{j} - \slashed{\nabla}_j \slashed{\nabla}_i\right)u = F_{ij}\cdot u\label{vcurv}
\end{align}
and is likewise compatible with the inner product $\left<\cdot,\cdot\right>$ on $V$, viz. for all $i\in\left\{1,\dots,n\right\}$ and $V$-valued functions $u,v$,
\begin{align}
	\partial_{i}\left<u,v\right>=\left<\slashed{\nabla}_{i}u,v\right> + \left<u,\slashed{\nabla}_{i}v\right>.\label{vcompat}
\end{align}

A pair $(A,u)$ consisting of a gauge field $A$ and scalar field $u$ is said to be a \textit{Yang-Mills-Higgs pair} with potential $W\in C^{1}(\mathbb{R},\mathbb{R}^{+})$ if the equations
\begin{equation}
	\begin{split}
		-\sum_{i=1}^{n}\nabla_{i}F_{ij} + u\odot \slashed{\nabla}_{j}u &=0\\
	-\sum_{i=1}^{n}\slashed{\nabla}_{i}^{2}u + 2W'(|u|^{2})u&=0\end{split}\label{ymhe}\tag{YMHE}
\end{equation}
hold on $\mathbb{R}^{n}$, where $\odot:V\times V\rightarrow \mathfrak{g}$ is the unique bilinear form satisfying the relation
\begin{align}
	\left<X,u_{1}\odot u_{2}\right> = \left<X\cdot u_{1}, u_{2}\right>\label{dotadj}
\end{align}
for all $X\in \mathfrak{g}$ and $u_{1},u_{2}\in V$. The equations (\ref{ymhe}) arise as the Euler-Lagrange equations of the energy density
\begin{align*}
	e(A,u)&= \frac{1}{2}\left(\sum_{i<j}|F_{ij}|^2 + \sum_i |\slashed{\nabla}_i u|^2  \right) + W\circ |u|^2.
\end{align*}
A key feature of this energy density and in fact the system (YMHE) is that they are invariant under \textit{gauge transformations} $g:\mathbb{R}^{n}\rightarrow G$ which act on pairs $(A,u)$ according to the rules
\begin{align*}
g\cdot u &= \rho(g)u\\
\intertext{and}
g\cdot A_{i}&= \Ad_{g}(A_{i}) - \partial_{i} g \cdot g^{-1},
\end{align*}
for all $i\in\{1,\dots,n\}$, where in the last line $\cdot$ denotes right translation in the tangent bundle of $G$; explicitly, $e(g\cdot A,g\cdot u)=e(A,u)$ and if $(A,u)$ is a Yang-Mills-Higgs pair, so is $(g\cdot A, g\cdot u)$.

Yang-Mills-Higgs pairs were introduced by Higgs \cite{MR0175554} as a generalization of \textit{Yang-Mills fields} (pairs of the form $(A,0)$ with $W\equiv 0$). They were subsequently studied in a mathematical context by Taubes and others (cf. \cite{taubes1985} and the references therein).

In studying (\ref{ymhe}), it is natural to consider the corresponding flow: A smooth one-parameter family of pairs $\{(A(\cdot,t),u(\cdot,t)\}_{t\in\left[a,b\right[}$ is said to evolve by the \textit{Yang-Mills-Higgs flow} if the equations
\begin{equation}
	\begin{split}\partial_{t}A_{j}&=\sum_{i=1}^{n}\nabla_{i}F_{ij} - u\odot \slashed{\nabla}_{j}u\\
	\partial_{t}u&=\sum_{i=1}^{n}\slashed{\nabla}_{i}^{2}u - 2(W'\circ|u|^{2})\cdot u\end{split}\label{ymhf}\tag{YMHF}
\end{equation}
hold on $\mathbb{R}^{n}\times\left]a,b\right[$. This flow was first introduced by Hassel \cite{Hassell1993431} for $n=3$ and subsequently studied by various others in more general geometric settings (cf. \cite{MR3143906} and the references therein). As with the Yang-Mills-Higgs equations, the Yang-Mills-Higgs flow is also invariant with respect to gauge transformations provided that they do not depend on $t$. Moreover, if $W\equiv 0$ and a one-parameter family of pairs of the form $\{(A(\cdot,t),0)\}_{t\in\left]a,b\right[}$ satisfies the equations (\ref{ymhf}), we say that $A$ evolves by the \textit{Yang-Mills flow}; this flow was first suggested by Atiyah and Bott \cite{MR702806} and subsequently studied by various others, ultimately motivating the consideration of the Yang-Mills-Higgs flow.

A crucial tool in the study of the long-time behaviour of solutions to (\ref{ymhf}) with $b=\infty$ in dimensions greater than $4$ (cf. \cite{MR2099188}) is a monotonicity formula due to Hong \cite{hong1998monotonicity}, a generalization of a monotonicity formula for the Yang-Mills flow due to Chen and Shen \cite{cs1994monotonicity} which was in turn motivated by one for the harmonic map heat flow due to Struwe \cite{struwe1988}; all of these formul\ae\ are akin to the identity
	\begin{align*}
		\frac{\D}{\D t}\int_{\mathbb{R}^{n}}v(x,t)\cdot\Gamma_{(X,T)}(x,t)\D x=0
	\end{align*}
	satisfied by solutions $v:\mathbb{R}^{n}\times\mathbb{R}^{+}\rightarrow\mathbb{R}$ to the heat equation of appropriate growth at infinity, where $(X,T)\in\mathbb{R}^{n}\times\mathbb{R}^{+}$ and
	\begin{align*}
		\Gamma_{(X,T)}(x,t)=\frac{1}{\left(4\pi(T-t) \right)^{n/2}}\exp\left(\frac{|x-X|^{2}}{4(t-T)} \right).
	\end{align*}
If $\{(A(\cdot,t),u(\cdot,t))\}_{t\in\left]a,b\right[}$ evolves by the Yang-Mills-Higgs flow and is of appropriate growth at infinity, Hong's formula takes the form
		\begin{equation}
			\begin{split}
			&\frac{\D}{\D t}\left((4\pi(T-t))^{2}\int_{\mathbb{R}^{n}}e(A,u)(x,t)\Gamma_{(X,T)}(x,t)\D x  \right)\\
			&=-(4\pi(T-t)^{2})\int_{\mathbb{R}^{n}}\sum_{j}\left|\partial_{t}A_{j} + \sum_{k}\frac{(x-X)^{k}}{2(t-T)}F_{kj}\right|^{2} + \left|\partial_{t}u + \sum_{k}\frac{(x-X)^{k}}{2(t-T)}\slashed{\nabla}_{k}u\right|^{2}\D x\\
			&\hskip 50mm +4\pi(T-t)^{2}\int_{\mathbb{R}^{n}}\frac{\sum_{j}|\slashed{\nabla}_{j}u|^{2} + 4W\circ |u|^{2} }{2(t-T)}\D x\end{split}\label{hongmon}
		\end{equation}
		on $\left]a,T\right[$ for $T\in\left]a,b\right]$ and $X\in\mathbb{R}^{n}$.
		
		In contrast to this formula, monotonicity formul\ae\ for (almost) elliptic problems--- such as that for harmonic maps due to Schoen and Uhlenbeck \cite{MR664498}, that for Yang-Mills connections due to Price \cite{price1983monotonicity} and that for minimal surfaces due to Allard \cite{MR0307015}--- tend to be local in nature, being akin to the mean-value property of solutions to Laplace's equation. Less widely known is that an analogous local formula exists for solutions to the heat equation \cite{fulks1966mean,MR0315289}: If $E_{r}(X,T)=\left\{\Gamma_{(X,T)}>\frac{1}{r^{n}}\right\}$ with $r\in\left]0,\sqrt{4\pi T}\right[$ and $(X,T)\in\mathbb{R}^{n}\times\mathbb{R}^{+}$, and $v\in C^{2}(\mathbb{R}^{n}\times\mathbb{R}^{+})$ solves the heat equation, then
\begin{align*}
	\frac{\D}{\D r}\left(\frac{1}{r^{n}}\iint_{E_{r}(X,T)}\frac{|x-X|^{2}}{4(t-T)^{2}}v(x,t)\D x\D t\right)=0.
\end{align*}
Motivated by this formula, it was shown by Ecker \cite{MR1865979,ecker2005local} that analogues of the minimal surface and harmonic map monotonicity formul\ae\ modelled on this formula may be established for the mean curvature flow, harmonic map heat flow and a certain class of reaction-diffusion equations provided the \textit{heat ball} $E_{r}(X,T)$ is appropriately modified. Moreover, it was shown that, when evaluated on special solutions of each of the aforementioned flows, the local quantity co\"incides with the global one.

The purpose of this paper is to establish a local analogue of Hong's monotonicity formula for the Yang-Mills-Higgs flow, first concentrating as in \cite{ecker2005local} on certain self-similar solutions, then establishing a local monotonicity formula more generally.

\vskip 2mm
\noindent{\small\textit{Acknowledgements.} This research was mostly carried out as part of the author's doctoral thesis at the Free University of Berlin under the supervision of Klaus Ecker, to whom much gratitude is due. The author gratefully acknowledges financial support from the Max Planck Institute for Gravitational Physics and the Leibniz Universit\"at Hannover.}
\chapter{Scaling properties, heat balls and self-similar solutions}\label{scaling}
We first recall some facts pertaining to weighted backward heat kernels and heat balls from \cite{ecker2005local} in a form suitable for our purposes. Fix $(X,T)\in\mathbb{R}^{n}\times\mathbb{R}$ and let $\bhk_{(X,T)}:\mathbb{R}^{n}\times\left]-\infty,T\right[\rightarrow\mathbb{R}^{+}$ be defined by
\begin{align*}
	\bhk_{(X,T)}(x,t)=\frac{1}{\left(4\pi(T-t)\right)^{\frac{n-4}{2}}}\exp\left(\frac{|x-X|^{2}}{4(t-T)} \right).
\end{align*}
For each $r>0$, we introduce the so-called \textit{weighted} heat ball of radius $r$ centred at $(X,T)$ by
\begin{align*}
	\hb{r}(X,T)&=\left\{(x,t)\in\mathbb{R}^{n}\times\left]-\infty,T\right[:\bhk_{(X,T)}(x,t)>\frac{1}{r^{n-4}}\right\}.
\end{align*}
A quick computation shows that in fact,
\begin{align*}
	\hb{r}(X,T)&=\bigcup_{t\in\left]T-\frac{r^{2}}{4\pi},T\right[}B_{R_{r}(t-T)}(X)\times\{t\}
\end{align*}
with $R_{r}(\tau)=\sqrt{2(n-4)\tau\log\left(-\frac{4\pi\tau}{r^{2}}\right)}$ for $\tau\in\left]-\frac{r^{2}}{4\pi},0\right[$. Note that $R_{r}(\tau)\leq c_{n}r$ with $c_{n}=\sqrt{\frac{n-4}{2\pi e}}$.

One useful property of these heat balls is that the integral of an appropriately scale-invariant function weighted against $\bhk$ on $\mathbb{R}^{n}$ may be written directly in terms of its integral on a heat ball with an appropriate weight function. This is the content of the following proposition.
\begin{prop}[{\cite[Proposition 1.5]{ecker2005local}}]\label{scalingprop} Suppose $f:\mathbb{R}^{n}\times\left]-\infty,T\right[\rightarrow\mathbb{R}^{+}$ is a measurable function such that $f(X+r(x-X),T+r^{2}(t-T))=r^{-4}f(x,t)$ for all $r>0$ and $(x,t)\in \mathbb{R}^{n}\times\left]-\infty,T\right[$. Then for all $t\in\left]-\infty,T\right[$ and $r>0$,
\begin{align*}
	\int_{\mathbb{R}^{n}}f(x,t)\cdot\bhk_{(X,T)}(x,t)\D x=\frac{1}{r^{n-4}}\iint_{\hb{r}(X,T)}f(y,s)\cdot\frac{n-4}{2(T-s)}\D y\D s.
\end{align*}
\end{prop}
We now turn our attention to one-parameter families of pairs $\{(A(\cdot,t),u(\cdot,t))\}_{t\in\left]-\infty,T\right[}$. Define for $r>0$ the rescaled family of pairs $\{(A^{r}(\cdot,t),u^{r}(\cdot,t))\}_{t\in\left]-\infty,T\right[}$ by
		\begin{align}
			A_{i}^{r}(x,t)&=r A_{i}(X + r(x-X),T + r^{2}(t-T))\ \forall i\in\{1,\dots,n\}\ \textup{and}\nonumber\\
			u^{r}(x,t)&= u(X + r(x-X),T + r^{2}(t-T)).\label{scaledpair}
		\end{align}
		Now, \textit{self-similar solutions} about $(X,T)$ are characterised by the condition that $A^{r}_{i}\equiv A_{i}$ for all $i\in\{1,\dots,n\}$ and $u^{r}\equiv u$ for all $r>0$. By differentiating these equations at $r=1$ and dividing through by $2(t-T)$, we obtain the identities
		\begin{align}
			\partial_{t}A_{i}(x,t) + \frac{1}{2(t-T)}A_{i}(x,t) + \sum_{k=1}^{n}\frac{(x-X)^{k}}{2(t-T)}\partial_{k}A_{i}(x,t)&=0\\
			\intertext{and}
			\partial_{t}u(x,t) + \sum_{k=1}^{n}\frac{(x-X)^{k}}{2(t-T)}\partial_{k}u(x,t)&=0
		\end{align}
		for all $i\in\{1,\dots,n\}$ and $(x,t)\in\mathbb{R}^{n}\times\left]-\infty,T\right[$, which also characterise self-similarity. Better still, these equations may be cast in the form
		\begin{align*}
			\mathcal{S}^{A}_{i}(x,t):=\partial_{t}A_{i}(x,t) + \sum_{k=1}^{n}\frac{(x-X)^{k}}{2(t-T)}F_{ki}(x,t)&=0\\
			\intertext{and}
			\mathcal{J}^{u}(x,t):=\partial_{t}u(x,t) + \sum_{k=1}^{n}\frac{(x-X)^{k}}{2(t-T)}\slashed{\nabla}_{k}u(x,t)&=0
		\end{align*}
		after passing to a \textit{radial gauge} in which $\sum_{i}(x-X)^{i}A_{i}(x,t)=0$ for all $(x,t)\in\mathbb{R}^{n}\times\left]-\infty,T\right[$  (cf. \cite{MR2034580}). Thus, it may be read off from (\ref{hongmon}) that if $\{(A(\cdot,t),u(\cdot,t))\}_{t\in\left]-\infty,T\right[}$ evolves by the Yang-Mills-Higgs flow, is of appropriate growth at infinity, self similar about $(X,T)$ and the conditions $W\circ |u|^{2}=0$ and $\slashed{\nabla}_{i}u=0$ hold, then $\int_{\mathbb{R}^{n}}e(A,u)(x,t)\cdot\bhk_{(X,T)}(x,t)\D x$ is independent of $t$. A closer look at the Yang-Mills-Higgs energy density $e(A,u)$ implies more, however.
		
		Denoting the gauge-covariant derivative on scalar fields and the curvature induced by $A^{r}$ by $\slashed{\nabla}^{r}$ and $F^{r}$ respectively, we see that
		\begin{equation}
			\begin{aligned}
			&e(A,u)(X+r(x-X),T+r^{2}(t-T))\\
			&\qquad=\left(r^{-4}\cdot\frac{1}{2}\sum_{i<j}|F_{ij}^{r}|^{2} + r^{-2}\cdot\frac{1}{2}\sum_{i}|\slashed{\nabla}^{r}_{i}u^{r}|^{2} + W\circ |u^{r}|^{2}\right)(x,t).
		\end{aligned}\label{edident}
		\end{equation}
		\begin{prop}\label{speciallocmon} Suppose the one-parameter family of pairs $\{(A(\cdot,t),u(\cdot,t))\}_{t\in\left]-\infty,T\right[}$ is such that $A$ is self similar about $(X,T)$, viz. $A^{r}\equiv A$ for all $r>0$, $\slashed{\nabla}_{i}u\equiv 0$ for all $i\in\{1,\dots,n\}$ and $W\circ |u|^{2}\equiv 0$. Then for all $t\in\left]-\infty,T\right[$ and $r>0$,
	\begin{align*}
		\int_{\mathbb{R}^{n}}e(A,u)(x,t)\cdot\bhk_{(X,T)}(x,t)\D x=\frac{1}{r^{n-4}}\iint_{\hb{r}(X,T)}e(A,u)(y,s)\cdot\frac{n-4}{2(T-s)}\D y\D s.
	\end{align*}
		\end{prop}
		\begin{proof} We note first that since $A$ is self similar about $(X,T)$, $F^{r}\equiv F$ for all $r>0$. Hence, by (\ref{edident}),
			\begin{align*}
				e(A,u)(X+r(x-X),T+r^{2}(t-T))=r^{-4}\cdot\frac{1}{2}\sum_{i<j}|F_{ij}|^{2}(x,t)=r^{-4}e(A,u)(x,t).
			\end{align*}
			Proposition \ref{scalingprop} then implies the claim. 
		\end{proof}
\chapter{Local monotonicity more generally}
Fix $a,T\in\mathbb{R}$ and $X\in\mathbb{R}^{n}$ and let $c_{n}$, $\mathcal{S}^{A}_{i}$ and $\mathcal{J}^{u}$ be as in \S\ref{scaling}. We first begin with a lemma that shall guarantee the finiteness of the singular integrals occurring in the local monotonicity formula; this should be compared with \cite[Appendix]{ecker2005local}.
\begin{lem}\label{estimates} If $\{(A(\cdot,t),u(\cdot,t))\}_{t\in\left[a,T\right[}$ evolves by the Yang-Mills-Higgs flow, then for all $r\in\left]0,\sqrt{4\pi (T-a)}\right[$ the estimates
	\begin{equation}
		\begin{split}
		&\iint_{\hb{r}(X,T)}\left(\sum_{j=1}^{n}\left|\mathcal{S}^{A}_{j}\right|^{2} + \left|\mathcal{J}^{u}\right|^{2}\right)(x,t)\D x\D t\\
		&\hskip 10mm \leq 2\widetilde{c}(n,\eta)r^{-2}\int_{T-\frac{r^{2}}{4\pi}}^{T}\int_{B_{2c_{n}r}(X)}e(A,u)(x,t)\D x\D t\\
		&\hskip 60mm +2\int_{B_{2c_{n}r}(X)}e(A,u)(x,T-\frac{r^{2}}{4\pi})\D x\end{split}\label{scalebound}
	\end{equation}
	and
	\begin{equation}
		\begin{split}
			&\frac{1}{R_{r}(t-T)^{n-4}}\int_{B_{R_{r}(t-T)}(X)}e(A,u)(x,t)\D x\\
		&\hskip 10mm \leq (4\pi)^{\frac{n-4}{2}}\exp(\frac{1}{4})\left(\frac{\widetilde{c}(n,\eta)}{r^{n-2}}\int_{T-\frac{r^{2}}{4\pi}}^{T}\int_{B_{2c_{n}r}(X)}e(A,u)(x,t)\D x\D t\right.\\
		&\hskip 60mm \left. + \frac{1}{r^{n-4}}\int_{B_{2c_{n}r}(X)}e(A,u)(x,T-\frac{r^{2}}{4\pi})\D x\right)\end{split}\label{energybound}
	\end{equation}
	hold, the latter for $t\in\left]T-e^{-\frac{1}{2(n-4)}}\frac{r^{2}}{4\pi},T\right[$, where $\widetilde{c}(n,\eta)$ is a constant depending only on $n$ and a smooth function $\eta:\mathbb{R}\rightarrow\left[0,1\right]$ with $\left.\eta\right|_{\left]-\infty,\frac{1}{2}\right[}\equiv 0$ and $\left.\eta\right|_{\left[1,\infty\right[}\equiv 1$.
\end{lem}
\begin{proof}
It may be shown using the methods of \cite{hong1998monotonicity} that
\begin{equation}
	\begin{split}&\frac{\D}{\D t}\int_{\mathbb{R}^{n}}\left(e(A,u)\cdot\bhk_{(X,s)}\cdot\varphi^{2}\right)(x,t)\D x\\
	&=-\int_{\mathbb{R}^{n}}\left(\left[\sum_{j}\left|\mathcal{S}^{A}_{j}\right|^{2} + \left|\mathcal{J}^{u}\right|^{2}\right]\bhk_{(X,s)}\varphi^{2}\right)(x,t)\D x\\
       &\qquad -\int_{\mathbb{R}^{n}}\left(\frac{\sum_{i}|\slashed{\nabla}_{i}u|^{2} + 4W\circ |u|^{2}}{2(s-t)}\cdot\bhk_{(X,s)}\varphi^{2}\right)(x,t)\D x\\
	&\qquad +2\int_{\mathbb{R}^{n}}\left(\bhk_{(X,s)}\cdot\varphi\cdot e(A,u)\cdot\left[\partial_{t}\varphi + \sum_{k}\frac{(x-X)^{k}}{2(t-s)}\partial_{k}\varphi\right]\right)(x,t)\D x\\
	&\qquad -2\int_{\mathbb{R}^{n}}\left(\bhk_{(X,s)}\cdot\varphi\cdot\left[\sum_{j}\left<\sum_{k}\partial_{k}\varphi\cdot F_{kj},\mathcal{S}^{A}_{j}\right>+\left<\sum_{k}\partial_{k}\varphi\cdot\slashed{\nabla}_{k}u,\mathcal{J}^{u}\right>\right]\right)(x,t)\D x\end{split}\label{cutoffmon}
\end{equation}
holds on $\left]0,T\right[$, where $s\geq T$, $\varphi\in C^{\infty}(\mathbb{R}^{n}\times\mathbb{R},\mathbb{R})$ is such that $\varphi(\cdot,t)\in C_{0}^{\infty}(\mathbb{R}^{n})$ for each $t\in\mathbb{R}$. We first take $\varphi$ to be such that
\begin{align*}
	\varphi(x,t)=\begin{cases}\eta(\frac{c_{n}r}{|x-X|}),&x\neq 0\\ 1,&x=0  \end{cases}
\end{align*}
with $\eta$ and $r$ as in the statement of the lemma and $c_{n}$ as in \S\ref{scaling}. It is then clear that for all $t\in\mathbb{R}$, $\chi_{B_{R_{r}(t-T)}(X)}\leq\chi_{B_{c_{n}r}(X)}\leq \varphi(\cdot,t)\leq \chi_{B_{2c_{n}r}(X)}$, $|\nabla\varphi|(\cdot,t)\leq \frac{||\eta'||_{\infty}}{c_{n}r}\cdot\chi_{B_{2c_{n}r}(X)\backslash B_{c_{n}r}(X)}$ and $\partial_{t}\varphi\equiv 0$, where the $\chi_{\cdot}$ are characteristic functions and $||\eta'||_{\infty}=\sup |\eta'|$.

Now, by the Cauchy-Schwarz inequality and Young's inequality, it is clear that
\begin{align*}
	\left|\varphi\cdot\left[\sum_{j}\left<\sum_{k}\partial_{k}\varphi\cdot F_{kj},\mathcal{S}^{A}_{j}\right>\right]\right|\leq |\nabla\varphi|^{2}\cdot\sum_{i<j}|F_{ij}|^{2} + \frac{1}{4}\varphi^{2}\sum_{j}\left|\mathcal{S}^{A}_{j}\right|^{2}
\end{align*}
and
\begin{align*}
	\left|\varphi\cdot\left[\left<\sum_{k}\partial_{k}\varphi\cdot\slashed{\nabla}_{k}u,\mathcal{J}^{u}\right>\right]\right|\leq |\nabla\varphi|^{2}\cdot\sum_{i}|\slashed{\nabla}_{i}u|^{2} + \frac{1}{4}\varphi^{2}\left|\mathcal{J}^{u}\right|^{2}
\end{align*}
so that incorporating these inequalities into (\ref{cutoffmon}), discarding the second integral on the right-hand side and applying the Cauchy-Schwarz inequality to the third integrand, we obtain
\begin{equation}
	\begin{split}
	&\frac{\D}{\D t}\int_{\mathbb{R}^{n}}e(A,u)\cdot\bhk_{(X,s)}\cdot\varphi^{2}\D x\\
	&\hskip 2mm \leq-\frac{1}{2}\int_{\mathbb{R}^{n}}\left(\sum_{j}\left|\mathcal{S}^{A}_{j}\right|^{2} + \left|\mathcal{J}^{u}\right|^{2}\right)\bhk_{(X,s)}\varphi^{2}\D x\\
	&\hskip 10mm +\int_{\mathbb{R}^{n}}\bhk_{(X,s)}\cdot e(A,u)\cdot\left[4|\nabla\varphi|^{2}+ 2\varphi\left|\frac{x-X}{2(t-s)}\right|\cdot\left|\nabla\varphi\right|\right]\D x.
\end{split}\label{simpcutoff}
\end{equation}
Using the bounds on $\varphi$ and $\nabla\varphi$ and the fact that
\begin{align*}
	\max\left\{\bhk_{(X,s)},\frac{r^{2}}{s-t}\bhk_{(X,s)}  \right\}\leq \frac{\const(n)}{r^{n-4}}
\end{align*}
on $\left(\mathbb{R}^{n}\backslash B_{c_{n}r}(X)\right)\times\left]-\infty,s\right[$, we see that
\begin{multline*}
	\bhk_{(X,s)}\cdot\left[4|\nabla\varphi|^{2} + 2\varphi\left|\frac{x-X}{2(t-s)}\right|\cdot|\nabla\varphi|\right]\\
	\leq \frac{\const(n)}{r^{n-2}}\cdot\left(4||\eta'||_{\infty}^{2} + 2c_{n}||\eta'||_{\infty} \right)\chi_{B_{2c_{n}r}(X)}=:\frac{\widetilde{c}(n,\eta)}{r^{n-2}}\chi_{B_{2c_{n}r}(X)}
\end{multline*}
whence, integrating (\ref{simpcutoff}) on $\left]T-\frac{r^{2}}{4\pi},t\right[$ with $t\in\left]T-\frac{r^{2}}{4\pi},T\right[$ and using the bounds on $\varphi$, we arrive at
\begin{multline*}
	\int_{\mathbb{R}^{n}}\left(e(A,u)\cdot\bhk_{(X,s)}\cdot\varphi^{2}\right)(x,t)\D x + \frac{1}{2}\int_{T-\frac{r^{2}}{4\pi}}^{t}\int_{\mathbb{R}^{n}}\left(\sum_{j}\left|\mathcal{S}^{A}_{j}\right|^{2} + \left|\mathcal{J}^{u}\right|^{2}\right)\bhk_{(X,s)}\varphi^{2}\D x\D t\\
	\leq \frac{\widetilde{c}(n,\eta)}{r^{n-2}}\int_{T-\frac{r^{2}}{4\pi}}^{t}\int_{B_{2c_{n}r}(X)}e(A,u)\D x\D t + \frac{1}{r^{n-4}}\int_{\mathbb{R}^{n}}e(A,u)(x,T-\frac{r^{2}}{4\pi})\D x,
\end{multline*}
where the bound $\bhk_{(X,s)}(x,T-\frac{r^{2}}{4\pi})\leq \frac{1}{r^{n-4}}$ was also used.

Now, by setting $s=T$, discarding the first term on the left-hand side, using the bound $\varphi(\cdot,t)\geq\chi_{B_{R_{r}(t-T)}(X)}$ and the fact that $\left.\bhk_{(X,T)}\right|_{B_{R_{r}(t-T)}(X)} > \frac{1}{r^{n-4}}$, we immediately obtain (\ref{scalebound}) after passing to the limit $t\nearrow T$, making use of the monotone convergence theorem in the process. On the other hand, a quick computation shows that for $t\in \left]T-e^{-\frac{1}{2(n-4)}}\frac{r^{2}}{4\pi},T\right[$,
\begin{align*}
	t + R_{r}(t-T)^{2} > T
\end{align*}
so that restricting our attention to such $t$, taking $s= t+ R_{r}(t-T)^{2}$, discarding the second term on the left-hand side, bounding the first term on the right-hand side by the integral with the same integrand over $B_{2c_{n}r}(X)\times\left]T-\frac{r^{2}}{4\pi},T\right[$ and using the fact that
\begin{align*}
	\left(\bhk_{(X,s)}\cdot\varphi^{2}\right)(x,t)\geq \frac{\exp(-\frac{1}{4})}{(4\pi)^{\frac{n-4}{2}}R_{r}(t-T)^{n-4}}\chi_{B_{R_{r}(t-T)}(X)}
\end{align*}
to bound the first term on the left-hand side from below, we immediately obtain (\ref{energybound}).
\end{proof}
Before proving the local monotonicity formula in greater generality, we recall the following integration-by-parts formula.
\begin{lem}[{\cite[Lemma 1.6]{ecker2005local}}]\label{euclheatball1} If $\xi\in C^{1}(\hb{r_{0}}(X,T),\mathbb{R}^{n})$ for some $r_{0}>0$, then
\begin{align*}
\iint_{\hb{r}(X,T)}(\Div\ \xi)(x,t)\D x\D t=-\frac{r}{n-4}\frac{\D}{\D r}\iint_{\hb{r}(X,T)}\xi(x,t)\cdot\frac{x-X}{2(t-T)}\D x\D t
\end{align*}
on $\left]0,r_{0}\right[$ whenever these integrals exist.
\end{lem}
\begin{thm}\label{locmon} If $\{(A(\cdot,t),u(\cdot,t))\}_{t\in\left[a,T\right[}$ evolves by the Yang-Mills-Higgs flow and $e(A,u)\in L^{1}\left(B_{2c_{n}r_{0}}(X)\times\left]T-\frac{r_{0}^{2}}{4\pi},T\right[\hskip 0.1mm \right)$ for some $r_{0}\leq\sqrt{4\pi(T-a)}$, then
	\begin{align*}
		&\frac{\D}{\D r}\left(\frac{1}{r^{n-4}}\iint_{\hb{r}(X,T)}e(A,u)(x,t)\cdot\frac{n-4}{2(T-t)}\right. \\
	       	&\hskip 20mm \left.\vphantom{\iint_{\hb{r}}\sum_{i=1}^{n}} - \sum_{j=1}^{n}\left<\sum_{i=1}^{n}\frac{(x-X)^{i}}{2(t-T)}F_{ij},\mathcal{S}^{A}_{j}\right>(x,t)
		- \left<\sum_{i=1}^{n}\frac{(x-X)^{i}}{2(t-T)}\slashed{\nabla}_{i}u,\mathcal{J}^{u}\right>(x,t)\D x\D t\right)\\
		&=\frac{n-4}{r^{n-3}}\iint_{\hb{r}(X,T)}\left(\frac{\sum_{i=1}^{n}|\slashed{\nabla}_{i}u|^{2} + 4W\circ |u|^{2}}{2(T-t)} + \sum_{j=1}^{n}|\mathcal{S}^{A}_{j}|^{2} + |\mathcal{J}^{u}|^{2}\right)(x,t)\D x\D t
	\end{align*}
	on $\left]0,r_{0}\right[$.
\end{thm}
\begin{rmk}\label{summrmk}
To see how the summability condition on $e(A,u)$ guarantees the finiteness of the integrals occurring in the theorem, note first that the latter two terms in the integrand on the right-hand side are clearly summable due to the estimate (\ref{scalebound}). On the other hand, the first term in the integrand on the right hand side may be bounded from above by $\frac{4e(A,u)(x,t)}{2(T-t)}$ and the latter two terms in the integrand on the left-hand side in modulus by
\begin{align*}
	\frac{|x-X|^{2}}{4(t-T)^{2}}e(A,u)(x,t) + \frac{1}{2}\left(\sum_{j=1}^{n}|\mathcal{S}^{A}_{j}|^{2} + |\mathcal{J}^{u}|^{2}\right)(x,t)
\end{align*}
whence, by the estimate (\ref{energybound}), we see that the finiteness of the left-hand integral is guaranteed by the finiteness of the integrals
\begin{align*}
	\int_{T-e^{-\frac{1}{2(n-4)}}\frac{r_{0}^{2}}{4\pi}}^{T}\frac{R_{r_{0}}(t-T)^{n-2}}{(t-T)^{2}}\D t
\end{align*}
and
\begin{align*}
\int_{T-e^{-\frac{1}{2(n-4)}}\frac{r_{0}^{2}}{4\pi}}^{T}\frac{R_{r_{0}}(t-T)^{n-4}}{T-t}\D t,
\end{align*}
since $\overline{\hb{r_{0}}(X,T)}\cap\left(\mathbb{R}^{n}\times\left[T-\frac{r_{0}^{2}}{4\pi}, T-e^{-\frac{1}{2(n-4)}}\frac{r_{0}^{2}}{4\pi}\right]\right)$ is a compact subset of the domain of $A$ and $u$. Altogether, we have an estimate of the form
\begin{multline*}
\frac{1}{r^{n-4}}\iint_{\hb{r}(X,T)}e(A,u)(x,t)\cdot\frac{n-4}{2(T-t)}\\
- \sum_{j=1}^{n}\left<\sum_{i=1}^{n}\frac{(x-X)^{i}}{2(t-T)}F_{ij},\mathcal{S}^{A}_{j}\right>(x,t) -\left<\sum_{i=1}^{n}\frac{(x-X)^{i}}{2(t-T)}\slashed{\nabla}_{i}u,\mathcal{J}^{u}\right>(x,t)\D x\D t\\
\leq \gamma(n,\eta)\left(\frac{1}{r^{n-2}}\int_{T-\frac{r^{2}}{4\pi}}^{T}\int_{B_{2c_{n}r}(X)}e(A,u)(x,t)\D x\D t\right.\\
		\left. + \frac{1}{r^{n-4}}\int_{B_{2c_{n}r}(X)}e(A,u)(x,T-\frac{r^{2}}{4\pi})\D x\right)
\end{multline*}
with $\gamma(n,\eta)$ a positive constant depending only on $n$ and $\eta$.
\end{rmk}
\begin{rmk} Note that the right-hand side of the local monotonicity formula vanishes precisely when $A$ is self similar about $(X,T)$ (in the appropriate gauge), $\slashed{\nabla}_{i}u\equiv 0$ for all $i\in\{1,\dots,n\}$, $W\circ |u|^{2}\equiv 0$ and $\partial_{t}u\equiv 0$ (cf. Proposition \ref{speciallocmon}); such conditions on $u$ hold e.g. if the following conditions are satisfied:
	\begin{enumerate}
		\item $W$ is a \textit{Higgs-like potential}, viz. $\mathcal{W}=\{x\in\mathbb{R}:W(x^{2})=W'(x^{2})=0\}\neq\emptyset$,
		\item$u$ is a parallel section, i.e. $\slashed{\nabla}_{i}u\equiv 0$ for all $i\in\{1,\dots, n\}$ and
		\item $|u|\in\mathcal{W}$.
	\end{enumerate}	
	In fact, if a solution to (YMHF) of the form $\{(0,u(\cdot,t)\}$ satisfies these three conditions, $u$ is then independent of $t$ and is said to be a \textit{Higgs equilibrium}. Thus, the right-hand side of the monotonicity formula vanishes on such solutions.
\end{rmk}
\begin{proof}[Proof of Theorem \ref{locmon}.] By translation invariance of the equations (YMHF), we may without loss of generality assume that $(X,T)=(0,0)$ and that $A(\cdot,t)$ and $u(\cdot,t)$ are defined for $t\in \left]a-T,0\right[$.\\
	We first proceed formally. By making the change of variables $(x,t)=(ry,r^{2}s)$, it is clear that	
	\begin{equation}
		\begin{split}
		&\frac{\D}{\D r}\left(\frac{1}{r^{n-4}}\iint_{\hb{r}}e(A,u)(x,t)\cdot\frac{n-4}{-2t}\D x\D t \right)\\
		&=(n-4)\frac{\D}{\D r}\iint_{\hb{1}}\frac{\left(\frac{1}{4}\sum_{i,j}|F_{ij}^{r}|^{2} + \frac{r^{2}}{2}\sum_{i=1}^{n}|\slashed{\nabla}_{i}^{r}u^{r}|^{2} + r^{4}W\circ|u^{r}|^{2}  \right)(y,s)}{-2s}\D y\D s
	\end{split}\label{firstdiff}\end{equation}
	where $\slashed{\nabla}^{r}$ and $F^{r}$ are as in \S\ref{scaling}. Note that
	\begin{equation}
		\begin{split}
		&\frac{\D}{\D r} \left(\frac{1}{4}\sum_{i,j}|F_{ij}^{r}|^{2} + \frac{r^{2}}{2}\sum_{i=1}^{n}|\slashed{\nabla}_{i}^{r}u^{r}|^{2} + r^{4}W\circ|u^{r}|^{2}  \right)(y,s)\cdot \frac{1}{-2s}\\
		&=\frac{r^{5}}{-2t}\left.\frac{\D}{\D r}\right|_{r=1}\left( \frac{1}{4}\sum_{i,j}|F_{ij}^{r}|^{2} + \frac{r^{2}}{2}\sum_{i=1}^{n}|\slashed{\nabla}_{i}^{r}u^{r}|^{2} + r^{4}W\circ|u^{r}|^{2}  \right)(x,t)
	\end{split}\label{scaledintegrand}
	\end{equation}
	We now proceed to relate this expression to the right-hand integrand in the theorem.

	Firstly, by (\ref{gcompat}),
\begin{align*}
	\left.\frac{\D}{\D r}\right|_{r=1}\left(\frac{1}{4}\sum_{i,j}|F_{ij}^{r}|^{2}\right)(x,t)&=\sum_{i,j}\left(|F_{ij}|^{2} + \frac{1}{2}\left<F_{ij},2t\partial_{t}F_{ij} + \sum_{k}x^{k}\nabla_{k}F_{ij}\right>\right)(x,t).
\end{align*}
It is clear from the definition of $F_{ij}$ that
\begin{align*}
	\partial_{t}F_{ij}&=\nabla_{i}(\partial_{t}A_{j})-\nabla_{j}\partial_{t}A_{i}
\end{align*}
and it follows immediately from (\ref{bianchi}) and the antisymmetry of $F_{ij}$ that
\begin{align*}
	\nabla_{k}F_{ij}&=\nabla_{i}F_{kj}+\nabla_{j}F_{ik}.
\end{align*}
Therefore, using the antisymmetry of $F_{ij}$, we see that
\begin{align*}
	\left.\frac{\D}{\D r}\right|_{r=1}\left(\frac{1}{4}\sum_{i,j}|F_{ij}^{r}|^{2}\right)(x,t)&=\sum_{i,j}\left(|F_{ij}|^{2} + \left<F_{ij},2t\nabla_{i}\partial_{t}A_{j} + \sum_{k}x^{k}\nabla_{i}F_{kj}\right> \right)(x,t)\\
	&=\Div\left(\sum_{i,j}\left<F_{ij},2t\mathcal{S}^{A}_{j}\right>e_{i}\right)(x,t)- \sum_{j}\left<\sum_{i}\nabla_{i}F_{ij},2t\mathcal{S}^{A}_{j}\right>(x,t),
\end{align*}
where (\ref{gcompat}) was used in the last line. Similarly, using (\ref{vcompat}), the definition of $\slashed{\nabla}_{i}u$ and (\ref{vcurv}), it is easy to see that
\begin{align*}
	&\left.\frac{\D}{\D r}\right|_{r=1}\left(\frac{r^{2}}{2}\sum_{i=1}^{n}|\slashed{\nabla}_{i}^{r}u^{r}|^{2}\right)(x,t)\\
	&=\sum_{i}\left(|\slashed{\nabla}_{i}u|^{2} + \left<\slashed{\nabla}_{i}u,2t\slashed{\nabla}_{i}u+\sum_{k}x^{k}\slashed{\nabla}_{k}u\slashed{\nabla}_{i}u\right>\right)(x,t)\\
	&=\sum_{i}\left(|\slashed{\nabla}_{i}u|^{2}+\left<\slashed{\nabla}_{i}u,2t\slashed{\nabla}_{i}\partial_{t}u+2t(\partial_{t}A_{i})\cdot u + \sum_{k}x^{k}\slashed{\nabla}_{i}\slashed{\nabla}_{k}u+x^{k}F_{ki}\cdot u\right>   \right)(x,t)\\
	&=\left(\Div\left(\sum_{i}\left<\slashed{\nabla}_{i}u,2t\mathcal{J}^{u}\right>e_{i}\right) + \sum_{i}\left<u\odot\slashed{\nabla}_{i}u,2t\mathcal{S}^{A}_{i}\right>- \left<\sum_{i}\slashed{\nabla}_{i}^{2}u,2t\mathcal{J}^{u}\right> + \sum_{i}|\slashed{\nabla}_{i}u|^{2}\right)(x,t),
\end{align*}
where in the last line (\ref{dotadj}) was also used. Finally, it is clear from (\ref{gcompat}) that
\begin{align*}
	\left.\frac{\D}{\D r}\right|_{r=1}\left(r^{4}W\circ |u^{r}|^{2}  \right)(x,t)&=\left(4W\circ|u|^{2} + \left<2(W'\circ|u|^{2})u,2t\mathcal{J}^{u}\right>\right)(x,t),
\end{align*}
whence, after making use of the equations (\ref{ymhf}), we see that (\ref{scaledintegrand}) is equal to
\begin{align*}
	r^{5}\cdot\left(\sum_{j}\left<\partial_{t}A_{j},\mathcal{S}^{A}_{j}\right> + \left<\partial_{t}u,\mathcal{J}^{u}\right> + \frac{\sum_{i}|\slashed{\nabla}_{i}u|^{2} + 4W\circ |u|^{2}}{-2t} -\Div\ Y\right)(x,t),
	\end{align*}
	where $\displaystyle Y=\sum_{i}\left(\sum_{j}\left<F_{ij},\mathcal{S}^{A}_{j}\right> + \left<\slashed{\nabla}_{i}u,\mathcal{J}^{u}\right>  \right)e_{i}$. Furthermore, after noting that
	\begin{align}
		Y(x,t)\cdot\frac{x}{2t}&=\left(\sum_{j}\left<\sum_{k}\frac{x^{k}}{2t}F_{kj},\mathcal{S}^{A}_{j}\right> + \left<\sum_{k}\frac{x^{k}}{2t}\slashed{\nabla}_{k}u,\mathcal{J}^{u}\right>\right)(x,t),\label{extraintegrand}
	\end{align}
	we may in fact write (\ref{scaledintegrand}) as
	\begin{align*}
		r^{5}\cdot\left(-\Div\ Y - Y\cdot\frac{x}{2t} + \frac{\sum_{i}|\slashed{\nabla}_{i}u|^{2} + 4W\circ |u|^{2}}{-2t} + \sum_{j}\left|\mathcal{S}^{A}_{j}\right|^{2} + \left|\mathcal{J}^{u}\right|^{2}\right)(x,t).
	\end{align*}
	Thus, interchanging integral and derivative in the equality (\ref{firstdiff}) and changing variables back, we obtain
	\begin{multline*}
	\frac{\D}{\D r}\left(\frac{1}{r^{n-4}}\iint_{\hb{r}}e(A,u)(x,t)\cdot\frac{n-4}{-2t}\D x\D t \right)\\
=\frac{n-4}{r^{n-3}}\iint_{\hb{r}}\left( \frac{\sum_{i}|\slashed{\nabla}_{i}u|^{2} + 4W\circ |u|^{2}}{-2t}+ \sum_{j}\left|\mathcal{S}^{A}_{j}\right|^{2} + \left|\mathcal{J}^{u}\right|^{2}-\Div\ Y - Y\cdot\frac{x}{2t}  \right)(x,t)\D x\D t
	\end{multline*}
	and, by Lemma \ref{euclheatball1},
	\begin{align*}
		\frac{n-4}{r^{n-3}}\iint_{\hb{r}}-(\Div\ Y)(x,t) - Y(x,t)\cdot\frac{x}{2t}\D x\D t&=\frac{\D}{\D r}\left(\frac{1}{r^{n-4}}\iint_{\hb{r}}Y(x,t)\cdot\frac{x}{2t}\D x\D t  \right)
	\end{align*}
whence, by (\ref{extraintegrand}), we have obtained the monotonicity formula, ending our formal computation.

A careful inspection shows that the preceding steps are valid for $r<r_{0}$ provided $\Div\ Y\in L^{1}\left(\hb{r_{0}}(0,0)\right)$, since all of the other terms are summable on $\hb{r_{0}}(0,0)$ by Lemma \ref{estimates} and Remark \ref{summrmk}; this holds e.g. if $A$ is smooth and can be smoothly continued past $t=0$. To drop this assumption, we instead compute analogously to before that
\begin{align*}
	&\frac{\D}{\D r}\left(\frac{1}{r^{n-4}}\iint_{\hb{r}}\left(e(A,u)(x,t)\cdot\frac{n-4}{-2t} - Y(x,t)\cdot \frac{x}{2t}\right)\eta_{k}(t)\D x\D t \right)\\
&=\frac{n-4}{r^{n-3}}\iint_{\hb{r}}\left(\vphantom{\sum_{j}\left|\sum_{k}\right| }\frac{\sum_{i}|\slashed{\nabla}_{i}u|^{2} + 4W\circ |u|^{2}}{-2t} + \sum_{j}\left|\mathcal{S}^{A}_{j}\right|^{2} + \left|\mathcal{J}^{u}\right|^{2}\right)(x,t)\cdot\eta_{k}(t)\D x\D t\\
       &\hskip 50mm+ \frac{n-4}{r^{n-3}}\iint_{\hb{r}}\frac{e(A,u)(x,t)}{-t}\cdot (-t)\eta_{k}'(-t)\D x\D t
\end{align*}
where $\eta_{k}(t):=\eta(2^{k}t)$ and $\eta$ is as in Lemma \ref{estimates}. Hence, after integrating this expression on $\left]r_{1},r_{2}\right[$ with $0<r_{1}<r_{2}<r_{0}$, using the fact that $(-t)\eta_{k}'(-t)\xrightarrow{k\rightarrow\infty}0$ and noting that, by Lemma \ref{estimates}, $\frac{e(A,u)}{-t}$ is summable on $\hb{r_{0}}(0,0)$, we may pass to the limit $k\rightarrow\infty$ and by standard integration theorems conclude that the last term vanishes, leaving us with the desired formula.
\end{proof}
Using the same techniques, it may be shown that if $(A,u)$ is a smooth Yang-Mills-Higgs pair, then
\begin{align*}
	\frac{\D}{\D r}\left(\frac{1}{r^{n-4}}\int_{B_{r}(X)}e(A,u)(x)\D x \right) &= r^{3-n}\int_{B_{r}(X)}\left(\sum_{i}|\slashed{\nabla}_{i}u|^{2} + 4W\circ |u|^{2}\right)(x)\D x\\
	&\hskip 2.5mm + r^{4-n}\int_{\partial B_{r}(X)}\left(\sum_{j}\left|\sum_{i}\nu^{i} F_{ij}\right| + \left|\sum_{i}\nu^{i}\slashed{\nabla}_{i}u\right|^{2}\right)(x)\D S_{x}
\end{align*}
for all $X\in\mathbb{R}^{n}$, where $\nu(x)=\frac{x-X}{r}$. This formula is well known in the case where $A$ is a Yang-Mills connection \cite{price1983monotonicity}; moreover, it is crucial in studying the compactness of suitable spaces of Yang-Mills connections in higher dimensions \cite{nakajima1988compactness}. It is hoped that Theorem \ref{locmon} may analogously yield information on the behaviour of Yang-Mills-Higgs flow pairs $\{(A(\cdot,t),u(\cdot,t)\}_{t\in\left]a,T\right[}$ for $t\nearrow T$ when $T$ is the maximal time of existence of the flow.

\end{document}